\documentclass[12pt, reqno]{amsart}
\usepackage{amsmath, amsthm, amscd, amsfonts, amssymb, graphicx, color}
\usepackage[bookmarksnumbered, colorlinks, plainpages]{hyperref}

\newtheorem{theorem}{Theorem}[section]

\newtheorem{proposition}[theorem]{Proposition}
\newtheorem{corollary}[theorem]{Corollary}
\theoremstyle{definition}
\newtheorem{definition}[theorem]{Definition}
\newtheorem{example}[theorem]{Example}

\theoremstyle{remark}

\numberwithin{equation}{section}

\begin{document}
\setcounter{page}{1}

%-------------------------- Pleased do not change the following line-------------------------------------------
%\noindent \textcolor[rgb]{0.99,0.00,0.00}{This is a submission to one of journals of TMRG: BJMA or AFA}\\[.5in]
%--------------------------------------------------------------------------------------------------------------

\title[Automatic continuity of generalized derivations]{Automatic continuity of new generalized derivations}
\author[A. Hosseini]{Amin Hosseini}
\address{Amin Hosseini\newline \indent Kashmar Higher Education Institute, Kashmar, Iran}
\email{\textcolor[rgb]{0.00,0.00,0.84}{hosseini.amin82@gmail.com, a.hosseini@kashmar.ac.ir}}

\author[C. Park]{Choonkil Park$^*$}
\address {Choonkil Park\newline \indent Department of Mathematics,
Research Institute for Convergence of Basic Science, Hanyang University, Seoul 04763, Korea}
\email{\textcolor[rgb]{0.00,0.00,0.84}{baak@hanyang.ac.kr}}

%\dedicatory{This paper is dedicated to Professor Madjid Mirzavaziri}



\date{Received: xxxxxx; Revised: yyyyyy; Accepted: zzzzzz.
\newline \indent $^{*}$ Corresponding author: Choonkil Park (email: baak@hanyang.ac.kr, fax: +82-2-2281-0019, orcid: 0000-0001-6329-8228)}

\begin{abstract} Let $\mathcal{A}$ and $\mathcal{B}$ be two algebras and let $n$ be a positive integer. A linear mapping $D:\mathcal{A} \rightarrow \mathcal{B}$ is called a \emph{strongly generalized derivation of order $n$} if there exist  families of linear mappings $\{E_k:\mathcal{A} \rightarrow \mathcal{B}\}_{k = 1}^{n}$,  $\{F_k:\mathcal{A} \rightarrow \mathcal{B}\}_{k = 1}^{n}$, $\{G_k:\mathcal{A} \rightarrow \mathcal{B}\}_{k = 1}^{n}$ and $\{H_k:\mathcal{A} \rightarrow \mathcal{B}\}_{k = 1}^{n}$ which satisfy $D(ab) = \sum_{k = 1}^{n}\left[E_k(a) F_k(b) + G_k(a)H_k(b)\right]$ for all $a, b \in \mathcal{A}$.
The purpose of this article is to study the automatic continuity of such derivations on Banach algebras and $C^{\ast}$-algebras.
\end{abstract} 
\maketitle

\baselineskip=13.8pt

\section{\bf Introduction and preliminaries}

Automatic continuity of derivations is an important topic in the theory of derivations and this topic has a fairly long history. The automatic continuity of a certain class of mappings, e.g.,  strongly generalized derivations, is the study of (algebraic) conditions on linear
maps on a category of Banach algebras, e.g., $C^{\ast}$-algebras, which guarantee that every strongly generalized derivation is continuous. Results on automatic continuity of linear maps defined on Banach algebras comprise a fruitful area of research intensively developed during the last sixty years. The reader is referred to \cite{bo, D, D1,  fh, V} for a deep and extensive study on this subject. %Mathematicians in their studies about continuous derivations have reached very fundamental and important conclusions about the range of derivations and we will mention some of them below.
%Throughout the paper, $\mathcal{A}$ and $\mathcal{B}$ denote two complex Banach algebras. If an algebra is unital, then $\textbf{1}$ stands for the identity element.
Let us recall some basic definitions and set the notations which are used in what follows. Let $\mathcal{A}$ be an algebra.
%The set of all primitive ideals of $\mathcal{A}$ is denoted by $\Pi(\mathcal{A})$. The Jacobson radical of an algebra $\mathcal{A}$ is the intersection of all primitive ideals of $\mathcal{A}$ which is denoted by $rad (\mathcal{A})$. Indeed, $rad(\mathcal{A}) = \bigcap_{\mathcal{P} \in \Pi(\mathcal{A})}\mathcal{P}$.  The algebra $\mathcal{A}$ is called semisimple if $rad(\mathcal{A}) = \{0\}$.
A nonzero linear functional $\varphi$ on $\mathcal{A}$ is called a \emph{character} if $\varphi(ab) = \varphi(a) \varphi(b)$ for all $a,b \in \mathcal{A}$. The set of all characters on $\mathcal{A}$ is denoted by $\Phi_\mathcal{A}$ and is called the character space of $\mathcal{A}$. According to \cite[Proposition 1.3.37]{D}, the kernel of $\varphi$, $\ker \varphi$, is a maximal ideal of $\mathcal{A}$ for every $\varphi \in \Phi_\mathcal{A}$. Recall that an algebra (or ring) $\mathcal{A}$ is called prime if for $a, b \in \mathcal{A}$, $a \mathcal{A} b = \{0\}$ implies that $a = 0$ or $b = 0$, and is semiprime if for $a\in \mathcal{A}$, $a \mathcal{A} a = \{0\}$ implies that $a = 0$.

Now we  introduce a new class of derivations called strongly generalized derivation of order $n$. Let $\mathcal{A}$ and $\mathcal{B}$ be two algebras and  $n$ be a positive integer. A linear mapping $D:\mathcal{A} \rightarrow \mathcal{B}$ is called a \emph{strongly generalized derivation of order $n$} if there exist families of linear mappings $\{E_k:\mathcal{A} \rightarrow \mathcal{B}\}_{k = 1}^{n}$,  $\{F_k:\mathcal{A} \rightarrow \mathcal{B}\}_{k = 1}^{n}$, $\{G_k:\mathcal{A} \rightarrow \mathcal{B}\}_{k = 1}^{n}$ and $\{H_k:\mathcal{A} \rightarrow \mathcal{B}\}_{k = 1}^{n}$ which satisfy $$D(ab) = \sum_{k = 1}^{n}\left[E_k(a) F_k(b) + G_k(a)H_k(b)\right]$$ for all $a, b \in \mathcal{A}$. Clearly, for $n = 1$, we have $D(ab) = E(a) F(b) + G(a)H(b)$ for all $a, b \in \mathcal{A}$, where $E,F,G, H:\mathcal{A} \rightarrow \mathcal{B}$ are linear mappings. If $D:\mathcal{A} \rightarrow \mathcal{B}$ is a strongly generalized derivation of order one associated with the mappings $E, F, G, H: \mathcal{A} \rightarrow \mathcal{B}$, then we say $D$ is an $(E, F, G, H)$-derivation. Also, if $D$ is a strongly generalized derivation of order $n$ associated with the families of linear mappings $\{E_k:\mathcal{A} \rightarrow \mathcal{B}\}_{k = 1}^{n}$,  $\{F_k:\mathcal{A} \rightarrow \mathcal{B}\}_{k = 1}^{n}$, $\{G_k:\mathcal{A} \rightarrow \mathcal{B}\}_{k = 1}^{n}$ and $\{H_k:\mathcal{A} \rightarrow \mathcal{B}\}_{k = 1}^{n}$, we say $D$ is an $\big(\{E_k\}_{k = 1}^{n}, \{F_k\}_{k = 1}^{n}, \{G_k\}_{k = 1}^{n}, \{H_k\}_{k = 1}^{n}  \big)$-derivation. As can be seen, if $D$ is a strongly generalized derivation of order one, then it covers the notion of derivation (if $D = E = H$ and $F = G = I$), the notion of generalized $(\sigma, \tau)$-derivation associated with a linear mapping $d$ (if $D = E$, $F = \sigma$, $G = \tau$ and $H = d$), the notion of left $\sigma$-centralizer (if $D = E$, $F = \sigma$ and $G$ or $H$ is zero), the notion of right $\tau$-centralizer (if $E$ or $F$ is zero, $G = \tau$ and $H = D$), the notion of generalized derivation associated with a mapping $d$ (if $D = E$, $F = G = I$ and $H = d$), the notion of homomorphism (if $D = E = F$ and $G = 0$ or $H = 0$), and the notion of ternary derivation (if $F = G = I$). Let $\mathcal{A}$ be an algebra. Recall that a triple of linear maps $(D, E, H)$ of $\mathcal{A}$ is a ternary derivation of $\mathcal{A}$ if $D(ab) = E(a)b + aH(b)$ for all $a, b \in \mathcal{A}$.

%Also, if $D$ is a strongly generalized derivation of order 2, we have \begin{align*} D(ab) = E_1(a) F_1(b) + G_1(a)H_1(b) + E_2(a) F_2(b) + G_2(a)H_2(b) \end{align*} for all $a, b \in \mathcal{A}$, where $E_i, F_i, G_i, H_i:\mathcal{A} \rightarrow \mathcal{B}$ are linear mappings for any $i \in \{1,2\}$.

Derivations are used in quantum mechanics (see \cite{B2, B3}), and it is interesting to note that the applications of generalized types of derivations, such as generalized derivations and $(\sigma, \tau)$-derivations, to important physical topics have been recently studied. See, for example, \cite{Hel} for the application of generalized derivations in general relativity, and \cite{E, H5} for the application of $(\sigma, \tau)$-derivations in theoretical physics. Therefore, it is possible that \emph{strongly generalized derivations of order $n$} might be considered by physicists in the future and used in the study of physical topics. Hence, it is interesting to investigate details of this

The main objective of this paper is to obtain some conclusions about the automatic continuity of strongly generalized derivations of order one on Banach algebras and $C^{\ast}$-algebras. Let us give a brief background in this regard. In 1958, Kaplansky \cite{K} conjectured that every derivation on a $C^{\ast}$-algebra is continuous. Two years later, Sakai \cite{S} answered this conjecture. Indeed, he proved that every derivation on a $C^{\ast}$-algebra is automatically continuous and later in 1972, Ringrose \cite{R}, by using the pioneering work of Bade and Curtis \cite{B1} concerning the automatic continuity of a module homomorphism between bimodules over $C(K)$-spaces, showed that every derivation from a $C^{\ast}$-algebra $\mathcal{A}$ into a Banach $\mathcal{A}$-bimodule is automatically continuous. Also, Johnson and Sinclair \cite{J} investigated the continuity of derivations on semisimple Banach algebras. In addition, in an interesting article, Peralta and Russo \cite{P} investigated automatic continuity of derivations on $C^{\ast}$-algebras and $JB^{\ast}$-triples. %Also, in \cite{m1}, it is shown that if $\sigma, \tau$ are continuous $\ast$-linear mappings, then every $(\sigma, \tau)$-derivation from a $C^{\ast}$-algebra into $B(\mathfrak{H})$ is automatically continuous, where $\mathfrak{H}$ is a Hilbert space, and in \cite{He} the assumption of linearity of $\sigma, \tau$ were deleted.
Moreover, Hou and Ming \cite{Hu} proved that if $\mathcal{X}$ is simple and $\sigma, \tau$ are surjective and continuous mappings on $B(\mathcal{X})$, then every $(\sigma, \tau)$-derivation on $B(\mathcal{X})$ is continuous. Recently, Hosseini \cite{H4} studied automatic continuity of $(\psi,\phi)$-derivations.  It is worth to note that there are some other results on automatic continuity in the literature. For example,  Johnson proved in \cite{J87} that every generalized $\ast$-homomorphism between $C^{\ast}$-algebras is continuous, while several extensions to generalized triple homomorphisms from a $C^{\ast}$-algebra or
from a $JB^{\ast}$-triple are established in \cite{GP13}. We further know from \cite{EP18} that every generalized derivation on a von Neumann algebra and every linear mapping on a von Neumann algebra which is a derivation or a triplet derivation at zero is automatically continuous. Furthermore, every generalized Jordan derivation from a $C^{\ast}$-algebra into a Jordan Banach module is continuous \cite{GH18}.

Now, we state some of the results in this paper.
Suppose that $\mathcal{A}$ and $\mathcal{B}$ are two complex algebras such that $\mathcal{A}$ is simple and unital with identity element $\textbf{1}$ , and $\Phi_{\mathcal{B}} \neq \phi$. Let $D: \mathcal{A} \rightarrow \mathcal{B}$ be a $(D, F, F, H)$-derivation such that $\varphi_0 (D(\textbf{1})) = \varphi_0 (H(\textbf{1}))\neq 0$ and $\varphi_0 (D(a_0)) \neq 0$ for some $a_0 \in \mathcal{A} \backslash \{\textbf{1}\}$ and some $\varphi_0 \in \Phi_{\mathcal{B}}$. Then $dim(\mathcal{A}) = 1$ and $dim(D(\mathcal{A})), dim(F(\mathcal{A})),$ $dim(H(\mathcal{A})) \leq 1$. In addition, if $\mathcal{A}$ and $\mathcal{B}$ are two normed algebras, then $D$, $F$ and $H$ are continuous (see Theorem \ref{9}).

In Theorem \ref{11} we also obtain the following:\\ Let $\mathcal{A}$ be a unital $C^{\ast}$-subalgebra of $B(\mathfrak{H})$, the $C^{\ast}$-algebra of all bounded linear operators on a Hilbert space  $\mathfrak{H}$, and let $D:\mathcal{A} \rightarrow B(\mathfrak{H})$ be an $\ast-(E, F, G, H)$-derivation, i.e. $D, E, F, G$ and $H$ are $\ast$-mappings. Then the following statements hold:\\
(i) If we suppose that $E(I) = H(I) = I$, where $I$ is the identity mapping on $\mathfrak{H}$, $H(a)D(I)E(b) = - E(a)D(I)H(b)$ for all $a, b \in \mathcal{A}$ and also $(E + H)$ is a $D$-continuous mapping, then $D$ is automatically continuous.\\
(ii) If we suppose that $F(I) = G(I) = I$, $G(a)D(I)F(b) = - F(a)D(I)G(b)$ for all $a, b \in \mathcal{A}$ and also $(F + G)$ is a $D$-continuous mapping, then $D$ is automatically continuous. Using the just commented theorem, we also obtain a result on the automatic continuity of generalized $(\sigma, \tau)$-derivations on $C^{\ast}$-algebras.

Moreover, we prove that if $\mathcal{A}$ is a Banach algebra, $\mathcal{B}$ is a simple Banach algebra and $D:\mathcal{A} \rightarrow \mathcal{B}$ is both a continuous $(D, F, G, H)$-derivation and continuous $(H, F, G, D)$-derivation such that $H$ is a continuous linear mapping and $F, G$ are surjective, then $F$ and $G$ are continuous or $H$ and $D$ are identically zero.

Some other results concerning the continuity and characterization of $(E, F, G, H)$-derivations are also discussed.

\section {\bf Main results}

%Throughout this section, $\sigma, \tau:\mathcal{A} \rightarrow (\mathcal{A} \ or \mathcal{B})$ are mappings and $I$ denotes the identity mapping on an algebra.
Throughout the paper, $I$ denotes the identity mapping on an algebra and $\textbf{1}$ stands for the identity element of any unital algebra.  We begin with the following definition.

\begin{definition}
Let $\mathcal{A}$ and $\mathcal{B}$ be two algebras and $n$ be a positive integer. A linear mapping $D:\mathcal{A} \rightarrow \mathcal{B}$ is called a \emph{strongly generalized derivation of order $n$} if there exist families $\{E_k:\mathcal{A} \rightarrow \mathcal{B}\}_{k = 1}^{n}$,  $\{F_k:\mathcal{A} \rightarrow \mathcal{B}\}_{k = 1}^{n}$, $\{G_k:\mathcal{A} \rightarrow \mathcal{B}\}_{k = 1}^{n}$ and $\{H_k:\mathcal{A} \rightarrow \mathcal{B}\}_{k = 1}^{n}$ of linear mappings which satisfy $$D(ab) = \sum_{k = 1}^{n}\left[E_k(a) F_k(b) + G_k(a)H_k(b)\right]$$ for all $a, b \in \mathcal{A}$. Clearly, for $n = 1$, we have $D(ab) = E(a) F(b) + G(a)H(b)$ for all $a, b \in \mathcal{A}$, where $E,F,G, H:\mathcal{A} \rightarrow \mathcal{B}$ are linear mappings.
\end{definition}

If $D:\mathcal{A} \rightarrow \mathcal{B}$ is a strongly generalized derivation of order one associated with the mappings $E, F, G, H: \mathcal{A} \rightarrow \mathcal{B}$, then we say $D$ is an $(E, F, G, H)$-derivation. Also, if $D$ is a strongly generalized derivation of order $n$ associated with the families $\{E_k:\mathcal{A} \rightarrow \mathcal{B}\}_{k = 1}^{n}$,  $\{F_k:\mathcal{A} \rightarrow \mathcal{B}\}_{k = 1}^{n}$, $\{G_k:\mathcal{A} \rightarrow \mathcal{B}\}_{k = 1}^{n}$ and $\{H_k:\mathcal{A} \rightarrow \mathcal{B}\}_{k = 1}^{n}$, we say $D$ is an $\big(\{E_k\}_{k = 1}^{n}, \{F_k\}_{k = 1}^{n}, \{G_k\}_{k = 1}^{n}, \{H_k\}_{k = 1}^{n}  \big)$-derivation. As stated in Introduction, concepts such as derivation, generalized derivation and generalized $(\sigma, \tau)$-derivation are examples of strongly generalized derivations of order one. Here, we give some examples of strongly generalized derivations of order $ n \geq 2.$ Mirzavaziri and Omidvar Tehrani \cite{mir} introduced the concept of a ($\delta, \varepsilon$)-double derivation as follows:\\
Let $\mathcal{A}$ be an algebra and let $\delta, \varepsilon : \mathcal{A} \rightarrow \mathcal{A}$ be linear mappings. A linear mapping $d : \mathcal{A} \rightarrow \mathcal{A}$ is called a ($\delta, \varepsilon$)-double derivation if
\begin{align*}
d(ab) = d(a) b + a d(b) + \delta(a) \varepsilon (b) + \varepsilon(a) \delta(b)
\end{align*}
for all $a, b \in \mathcal{A}$. By a $\delta$-double derivation we mean a ($\delta, \delta$)-double derivation. For more material about $(\delta, \varepsilon)$-double derivations, see \cite{H4, mir}.

\begin{example} 
Every $(\delta, \varepsilon)$-double derivation is a strongly generalized derivation of order 2.
\end{example}

\begin{example}
 Let $\mathcal{A}$ and $\mathcal{B}$ be two algebras. A sequence $\{d_n\}$ of linear mappings from $\mathcal{A}$ into $\mathcal{B}$ is called a higher derivation if $d_n(ab) = \sum_{k = 0}^{n}d_{n - k}(a) d_k(b)$ for all $a, b \in \mathcal{A}$ and all nonnegative integer $n$. Let $n$ be a positive integer and $\{d_n\}$ be a higher derivation. Then every $d_n$ is a strongly generalized derivation of order $m$ in which
\[m = \left\{{\begin{array}{*{20}{c}}
\begin{array}{l}
\frac{n + 2}{2} \\
\frac{n + 1}{2}
\end{array}&
\begin{array}{l}
n \ is \ even,\\
n \ is \ odd.
\end{array}
\end{array}} \right.\]
\end{example}

We now establish the following auxiliary result.

\begin{proposition}\label{6} Let $\mathcal{A}$ be a unital ring,  $\mathcal{B}$ be a ring which is a domain and  $D:\mathcal{A} \rightarrow \mathcal{B}$ be a $(D, F, F, H)-$derivation such that $D(\textbf{1}) = H(\textbf{1})$. If $D(a_0) \neq 0$ for some $a_0 \in \mathcal{A} \backslash \{\textbf{1}\}$ and $D(\textbf{1}) \neq 0$, then $ker (F) = ker (D) \subseteq ker (H)$ and further $ker (D)$ is an ideal of $\mathcal{A}$.
\end{proposition}

\begin{proof} First of all, we show that if $D$ is a $(D, F, F, H)-$derivation such that $D(\textbf{1}) = H(\textbf{1})$ and $D(a_0) \neq 0$ for some $a_0 \in \mathcal{A} \backslash \{\textbf{1}\}$, then $D(\textbf{1}) = 0$ if and only if $F(\textbf{1}) = \textbf{1}$. It is evident that if $F(\textbf{1}) = \textbf{1}$, then $D(\textbf{1}) = H(\textbf{1}) = 0$. Conversely, assume that $D(\textbf{1}) = 0$. So we have $D(a_0) = D(a_0)F(\textbf{1}) + F(a_0) H(\textbf{1}) =  D(a_0)F(\textbf{1})$. Since $\mathcal{B}$ is a domain and $D(a_0) \neq 0$, we get that $F(\textbf{1}) = \textbf{1}$. Hence, $D(\textbf{1}) = 0$ if and only if $F(\textbf{1}) = \textbf{1}$ and consequently, we deduce that $D(\textbf{1}) \neq 0$ if and only if $F(\textbf{1}) \neq \textbf{1}$.\\
For $a \in ker (F)$, we have $D(a) = D(a)F(\textbf{1}) + F(a)H(\textbf{1}) = D(a)F(\textbf{1})$, and since $\mathcal{B}$ is a domain and $F(\textbf{1}) \neq \textbf{1}$, it is deduced that $D(a) = 0$. This means that $a \in \ker (D)$. So $ker (F) \subseteq ker (D)$. Now, suppose that $a \in ker (D)$. Thus, $0 = D(a) = D(a) F(\textbf{1}) + F(a) H(\textbf{1}) = F(a) H(\textbf{1})$. Since $\mathcal{B}$ is a domain and $H(\textbf{1}) \neq 0$, we get that $F(a) = 0$. It means that $a \in \ker (F)$ and so $ker (D) \subseteq ker (F)$. Therefore, we obtain that $ker (F) = ker (D)$. For any $a \in ker (D)$, we have $0 = D(a) = D(\textbf{1}) F(a) + F(\textbf{1}) H(a) = F(\textbf{1}) H(a)$, and since $\mathcal{B}$ is a domain and $D(\textbf{1}) \neq 0$, we arrive at $H(a) = 0$, which  means that $ ker (D) \subseteq ker (H)$. Our next task is to show that $ker (D)$ is an ideal of $\mathcal{A}$. Let $a \in ker(D)$ and $b$ be an arbitrary element of $\mathcal{A}$. We know that $ker (F) = ker (D) \subseteq ker (H)$. So we have $$ D(a b) = D(a) F (b) + F(a) H(b) = 0,$$ which means that $ker (D)$ is a right ideal of $\mathcal{A}$. Also, we have $$ D(b a) = D(b) F (a) + F(b) H(a) = 0,$$ which means that $ker (D)$ is a left ideal of $\mathcal{A}$. Consequently, $ker (D)$ is an ideal of $\mathcal{A}$, as desired.
\end{proof}

\begin{theorem} \label{9} 
Suppose that $\mathcal{A}$ and $\mathcal{B}$ are two complex algebras such that $\mathcal{A}$ is unital and simple and also $\Phi_{\mathcal{B}} \neq \phi$. Let $D: \mathcal{A} \rightarrow \mathcal{B}$ be a $(D, F, F, H)$-derivation such that $\varphi_0 (D(\textbf{1})) = \varphi_0 (H(\textbf{1}))\neq 0$ and $\varphi_0 (D(a_0)) \neq 0$ for some $a_0 \in \mathcal{A} \backslash \{\textbf{1}\}$ and some $\varphi_0 \in \Phi_{\mathcal{B}}$. Then $dim(\mathcal{A}) = 1$ and $dim(D(\mathcal{A})), dim(F(\mathcal{A})), dim(H(\mathcal{A})) \leq 1$. In addition, if $\mathcal{A}$ and $\mathcal{B}$ are in addition two normed algebras, then $D$, $F$ and $H$ are continuous.
\end{theorem}

\begin{proof} Let $\varphi_0$ be a character on $\mathcal{B}$ such that $\varphi_0 (D(\textbf{1})) = \varphi_0 (H(\textbf{1}))\neq 0$ and $\varphi_0 (D(a_0)) \neq 0$ for some $a_0 \in \mathcal{A} \backslash \{\textbf{1}\}$. Let $\varphi_0 F = F_1$, $\varphi_0 H = H_1$ and $\varphi_0 D = D_1$. According to Proposition \ref{6}, $ker(D_1) = ker(F_1)$ and $ker(D_1)$ is an ideal of $\mathcal{A}$. Since $\mathcal{A}$ is simple, $ker(D_1) = \mathcal{A}$ or $ker (D_1) = \{0\}$ and since we are assuming that $D_1(a_0) \neq 0$ for some $a_0 \in \mathcal{A} \backslash \{\textbf{1}\}$, $ker(D_1) \neq \mathcal{A}$. So $ker (D_1) = \{0\}$ and this means that $D_1$ is injective. It is a well-known fact that every non-zero linear functional is surjective. Hence, $D_1: \mathcal{A} \rightarrow \mathbb{C}$ is surjective, and therefore $D_1$ is bijective, and clearly, $dim(\mathcal{A}) = 1$. It follows from \cite[Theorem 2.6-9]{Ke} that $dim(D(\mathcal{A})), dim(F(\mathcal{A})), dim(H(\mathcal{A})) \leq 1$. If $\mathcal{A}$ and $\mathcal{B}$ are in addition two normed algebras, then \cite[Theorem 2.7-8]{Ke} implies the continuity of $D$, $F$ and $H$.
\end{proof}

The concept of left (resp. right) continuity of a mapping was defined in \cite{He} and we state the definition here.

\begin{definition} \label{10} Let $\mathcal{A}$ and $\mathcal{B}$ be two normed algebras and  $T, S:\mathcal{A} \rightarrow \mathcal{B}$ be two mappings. The mapping $T$ is left (resp. right) $S$-continuous if $\lim_{x \rightarrow 0}T(x) S(b) = 0$ (resp. $\lim_{x \rightarrow 0}S(b) T(x) = 0$) for all $b \in \mathcal{B}$. If $T$ is both left and right $S$-continuous, then it is simply called $S$-continuous.
\end{definition}

\begin{theorem} \label{11} 
Let $\mathcal{A}$ be a unital $C^{\ast}$-subalgebra of the $C^{\ast}$-algebra $B(\mathfrak{H})$. Let $D:\mathcal{A} \rightarrow B(\mathfrak{H})$ be a $\ast-(E, F, G, H)$-derivation, i.e., $D, E, F, G$ and $H$ are $\ast$-mappings. The following  hold:\\
(i) Suppose $E(I) = H(I) = I$, where $I$ is the identity mapping on $\mathfrak{H}$, $H(a)D(I)E(b) = - E(a)D(I)H(b)$ for all $a, b \in \mathcal{A}$, and $(E + H)$ is a $D$-continuous mapping. Then $D$ is automatically continuous.\\
(ii) Suppose $F(I) = G(I) = I$, $G(a)D(I)F(b) = - F(a)D(I)G(b)$ for all $a, b \in \mathcal{A}$, and $(F + G)$ is a $D$-continuous mapping. Then $D$ is automatically continuous.
\end{theorem}

\begin{proof} (i) First note that $D(I) = F(I) + G(I)$. For any $a \in \mathcal{A}$, we have $$D(a) = E(a)F(I) + G(a)H(I) = E(a)F(I) + G(a).$$ So we have
\begin{align}\label{2.1}
G(a) = D(a) - E(a)F(I), \ \ \ \ \ \ \ \ \ \ \ \ (a \in \mathcal{A}).
\end{align}
Similarly, we can get
\begin{align}\label{2.2}
F(a) = D(a) - G(I)H(a), \ \ \ \ \ \ \ \ \ \ \ \ (a \in \mathcal{A}).
\end{align}
Using (\ref{2.1}) and (\ref{2.2}), we have the following:
\begin{align*}
D(ab) &= E(a)F(b) + G(a)H(b) \\ & = E(a)(D(b) - G(I)H(b)) + (D(a) - E(a)F(I))H(b) \\ & = E(a)D(b) - E(a) G(I)H(b) + D(a)H(b) - E(a)F(I)H(b) \\ & = D(a)H(b) + E(a)D(b) - E(a) D(I)H(b),
\end{align*}
which means that
\begin{align}\label{2.3}
D(ab) = D(a)H(b) + E(a)D(b) - E(a) D(I)H(b), \ \ \ \ \ \ \ (a, b \in \mathcal{A}).
\end{align}
Using (\ref{2.3}), we have the following:
\begin{align*}
D(ab) &= (D(ab)^{\ast})^{\ast} = \big(D(b^{\ast})H(a^{\ast}) + E(b^{\ast})D(a^{\ast}) - E(b^{\ast}) D(I)H(a^{\ast})\big)^{\ast} \\ & = H(a)D(b) + D(a)E(b) - H(a)D(I)E(b),
\end{align*}
which means that
\begin{align}\label{2.4}
D(ab) = H(a)D(b) + D(a)E(b) - H(a)D(I)E(b), \ \ \ \ \ \ \ (a, b \in \mathcal{A}).
\end{align}
Adding (\ref{2.3}) and (\ref{2.4}) and using the assumption that $H(a)D(I)E(b) = - E(a)D(I)H(b)$ for all $a, b \in \mathcal{A}$, we have 
\begin{align}\label{2.5}
2D(ab) = D(a)(H(b) + E(b)) + (H(a) + E(a))D(b), \ \ \ \ \ \ \ (a, b \in \mathcal{A}).
\end{align}
Considering $\Sigma = \frac{E + H}{2}$ in (\ref{2.5}), we get that
\begin{align*}
D(ab) = D(a)\Sigma(b) + \Sigma(a)D(b), \ \ \ \ \ \ \ \ \ \ \ \ \ \ \ \ (a, b \in \mathcal{A}).
\end{align*}
This means that $D$ is a $\Sigma$-derivation on $\mathcal{A}$ and since  $\Sigma = \frac{E + H}{2}$ is $D$-continuous, it follows from \cite[Theorem 4.4]{He} that $D$ is continuous, as desired.\\
(ii) This part is obtained in the same way as above.
\end{proof}

An immediate consequence of the above theorem about the continuity of generalized $\ast-(\sigma, \tau)$-derivations on $C^{\ast}$-algebras reads as follows:

\begin{corollary} 
 Let $\mathcal{A}$ be a unital $C^{\ast}$-subalgebra of the $C^{\ast}$-algebra $B(\mathfrak{H})$. Let $D:\mathcal{A} \rightarrow B(\mathfrak{H})$ be a generalized $\ast-(\sigma, \tau)$-derivation corresponding to a mapping $H:\mathcal{A} \rightarrow B(\mathfrak{H})$, i.e., $D(ab) = D(a) \sigma(b) + \tau(a) H(b)$ for all $a, b \in \mathcal{A}$. Suppose that $\sigma(I) = \tau(I) = I$, where $I$ is the identity mapping on $\mathfrak{H}$, $\tau(a)D(I)\sigma(b) = - \sigma(a)D(I)\tau(b)$ for all $a, b \in \mathcal{A}$, and $(\sigma + \tau)$ is a $D$-continuous mapping. Then $D$ is automatically continuous.
\end{corollary}

\begin{proof}
It follows from  Theorem \ref{11} (ii).
\end{proof}

\begin{theorem} 
Let $\mathcal{A}$ be a unital $C^{\ast}$-algebra and  $D:\mathcal{A} \rightarrow \mathcal{A}$ be an $(E, I, I, H)$-derivation. Then $D, E$ and $H$ are generalized derivations which are automatically continuous under any of the following:\\
(i) $D$ is a $\ast-(E, I, I, H)$-derivation, that is, $D, E$ and $H$ are $\ast$-mappings.\\
(ii) The $C^{\ast}$-algebra $\mathcal{A}$ is commutative.
\end{theorem}

\begin{proof} First note that the triplet $(D, E, H)$ is actually a ternary derivation. We have
$D(ab) = E(a) b + a H(b)$ for all $a, b \in \mathcal{A}$. Since  $D, E$ and $H$ are $\ast$-mappings, we have
\begin{align*}
D(ab) = D^{\ast}(ab)  = (D(ab)^{\ast})^{\ast} = \big(E(b^{\ast})a^{\ast} + b^{\ast}H(a^{\ast})\big)^{\ast} = %H^{\ast}(a) b + a E^{\ast}(b) \\ &
H(a) b + a E(b),
\end{align*}
which means that
\begin{align*}
D(ab) = H(a)b + aE(b), \ \ \ \ \ \ \ (a, b \in \mathcal{A}).
\end{align*}
Considering these two equalities for $D(ab)$, we have
\begin{align*}
D(ab)= E(a) b + a H(b)= H(a)b + aE(b), \ \ \ \ \ \ \ (a, b \in \mathcal{A}).
\end{align*}
%\begin{align*} D(ab) & = \frac{D(ab)}{2} + \frac{D(ab)}{2} \\ & =  \frac{E(a) b + a H(b)}{2} + \frac{H(a)b + aE(b)}{2} \\ & = \Omega(a) b + a \Omega(b), \end{align*} where $\Omega = \frac{E + H}{2}$.
According to \cite[Theorem 3.1]{H3}, there exists a derivation $\delta:\mathcal{A} \rightarrow \mathcal{A}$ such that $D = \delta + L_{D(\textbf{1})}$, $E = \delta + L_{E(\textbf{1})}$ and  $H = \delta + L_{H(\textbf{1})}$, where $L_a(b) = ab$ for all $a, b \in \mathcal{A}$. Hence, $D, E$ and $H$ are generalized derivations. It follows from the main theorem of \cite{S} that the derivation $\delta$ is continuous and so $D, E$ and $H$ are continuous as well.\\
(ii) Since $\mathcal{A}$ is commutative, one can easily get that
\begin{align*}
D(ab)= E(a) b + a H(b)= H(a)b + aE(b), \ \ \ \ \ \ \ (a, b \in \mathcal{A}).
\end{align*}
Arguing as the proof of (i), we obtain the required result.
\end{proof}

\begin{theorem} 
Let $\mathcal{A}$ be a unital algebra and  $D:\mathcal{A} \rightarrow \mathcal{A}$ be an $(E, I, I, H)$-derivation. Then $D$ and $H$ are generalized derivations. Furthermore, if we assume that every derivation on $\mathcal{A}$ is continuous, then $D, E$ and $H$ are continuous linear mappings.
\end{theorem}

\begin{proof} It is easy to see that $E(a) = D(a) - a H(\textbf{1})$ and $H(a) = D(a) - E(\textbf{1})a$ for all $a \in \mathcal{A}$. Considering these equations, we get $D(ab) = D(a)b + a D(b) - a D(\textbf{1}) b$ for all $a, b \in \mathcal{A}$. Define the mapping $\Delta: \mathcal{A} \rightarrow \mathcal{A}$ by $\Delta(a) = D(a) - D(\textbf{1})a$. We have
\begin{align*}
\Delta(ab) & = D(ab) - D(\textbf{1})ab \\ & = D(a)b + a D(b) - a D(\textbf{1}) b - D(\textbf{1})ab \\ & = (D(a) - D(\textbf{1}) a)b + a(D(b) - D(\textbf{1})b) \\ & = \Delta(a)b + a \Delta(b),
\end{align*}
which means that $\Delta$ is a derivation. Therefore, $D$ is a generalized derivation. Moreover, we have $H(a) = D(a) - E(\textbf{1})a = \Delta(a) + D(\textbf{1})a - E(\textbf{1})a = \Delta(a) + H(\textbf{1})a$, which means that $H$ is a generalized derivation. It is clear that the mappings $D, E$ and $H$ are continuous by hypothesis under which a derivation is continuous.
\end{proof}

Suppose that $\mathcal{Y}$ and $\mathcal{Z}$ are Banach spaces and $T: \mathcal{Y} \rightarrow \mathcal{Z}$ is a linear mapping. Recall that the set $$S(T) = \left\{z \in \mathcal{Z} \ : \ \exists \ \{y_n\} \subseteq \mathcal{Y} \ such \ that \ y_n \rightarrow 0, T(y_n) \rightarrow z\right\}$$ is called the separating space of $T$. By the Closed Graph Theorem, $T$ is continuous if and only if $S(T) = \{0\}$. For more material about the separating space and other results about this concept, see \cite{D, D1}.

\begin{theorem} \label{+++}
Suppose that $\mathcal{A}$ is a Banach algebra, $\mathcal{B}$ is a simple Banach algebra and $D:\mathcal{A} \rightarrow \mathcal{B}$ is both a continuous $(D, F, G, H)$-derivation and continuous $(H, F, G, D)$-derivation such that $H$ is continuous and $F, G$ are surjective. Then $F$ and $G$ are continuous or $H$ and $D$ are identically zero.
\end{theorem}

\begin{proof} First note that
\begin{align*}
D(ab) = D(a)F(b) + G(a)H(b) = H(a)F(b) + G(a)D(b), \ \ \ \ \ \ \ (a, b \in \mathcal{A}).
\end{align*}
Let $c \in S(F)$. Then there exists a sequence $\{c_n\} \subseteq \mathcal{A}$ such that $c_n \rightarrow 0$ and $F(c_n) \rightarrow c$. We have 
\begin{align}\label{2.6}
 0 = \lim_{n \rightarrow + \infty}D(a c_n) = \lim_{n \rightarrow + \infty}\left(D(a) F(c_n) + G(a) H(c_n)\right) = D(a) c
 \end{align} for all $a \in \mathcal{A}$.
%\begin{align} 0 = \lim_{n \rightarrow + \infty}\Delta(a c_n) = \lim_{n \rightarrow + \infty}\left(\gamma(a) \sigma(c_n) + \tau(a) \Delta(c_n)\right) = \gamma(a) c.\end{align}
Also, we have
\begin{align}\label{2.7}
 0 = \lim_{n \rightarrow + \infty}D(a c_n) = \lim_{n \rightarrow + \infty}\left(H(a) F(c_n) + G(a) D(c_n)\right) = H(a) c 
\end{align} for all $a \in \mathcal{A}$. Set $$\mathfrak{X} = \left\{c \in \mathcal{B} \ : \ D(a)c = 0 = H(a)c, \quad  \forall \ a \in \mathcal{A} \right\}.$$ According to the discussion above, $S(F) \subseteq \mathfrak{X}$. We shall show that $\mathfrak{X}$ is an ideal of $\mathcal{B}$. Obviously, if $c_1, c_2 \in \mathfrak{X}$ and $\lambda \in \mathbb{C}$, then $c_1 + \lambda c_2 \in \mathfrak{X}$. Now, suppose $c$ and $b$ are arbitrary elements of $\mathfrak{X}$ and $\mathcal{B}$, respectively. We will show that $cb \in \mathfrak{X}$. Let $a$ be an arbitrary element of $\mathcal{A}$. We have $D(a)(c b) = (D(a)c)b = 0$ and also $H(a)(c b) = (H(a)c)b = 0$, which means that $c b \in \mathfrak{X}$ and consequently, $\mathfrak{X}$ is a right ideal of $\mathcal{B}$. Since $F$ is a surjective mapping, there exists an element $a_1 \in \mathcal{A}$ such that $F(a_1) = b$. It follows from (\ref{2.6}) and (\ref{2.7}) that $$0 = D(a a_1)c = D(a) F(a_1)c + G(a) H(a_1)c = D(a) b c,$$ and further $$0 = D(a a_1)c = H(a) F(a_1)c + G(a) D(a_1)c = H(a) b c,$$ for all $a \in \mathcal{A}$. This means that $\mathfrak{X}$ is a left ideal and so it is an ideal of $\mathcal{B}$. Since $\mathcal{B}$ is a simple algebra, either $\mathfrak{X} = \{0\}$ or $\mathfrak{X} = \mathcal{B}$. Suppose that $\mathfrak{X} = \{0\}$. Since $S(F) \subseteq \mathfrak{X}$, we infer that $F$ is continuous. Now, suppose that $\mathfrak{X} = \mathcal{B}$, i.e., $\mathfrak{X} = \left\{c \in \mathcal{B} \ : \ D(a)c = 0 = H(a)c \ for \ all \ a \in \mathcal{A} \right\} = \mathcal{B}$. So $D(a) c = H(a)c = 0$ for all $a \in \mathcal{A}$, $c \in \mathcal{B}$. We can thus see that $D(a) \mathcal{B} D(a) = H(a) \mathcal{B} H(a) = \{0\}$ for all $a \in \mathcal{A}$. Since $\mathcal{B}$ is a simple algebra and every simple algebra is semiprime, we deduce that $D = H = 0$.

We turn to the continuity of $G$. Suppose that $g \in S(G)$. Then there exists a sequence $\{g_n\} \subseteq \mathcal{A}$ such that $g_n \rightarrow 0$ and $G(g_n) \rightarrow g$. We have 
\begin{align}\label{2.8}
0 = \lim_{n \rightarrow + \infty}D(g_n a) = \lim_{n \rightarrow + \infty}\left(H(g_n) F(a) + G(g_n) D(a)\right) = g D(a) \end{align} and also
\begin{align} \label{2.9}
0 = \lim_{n \rightarrow + \infty}D(g_n a) = \lim_{n \rightarrow + \infty}\left(D(g_n) F(a) + G(g_n) H(a)\right) = g H(a)
 \end{align}
for all $a \in \mathcal{A}$. Let $$\mathfrak{J} = \left\{g \in \mathcal{B} \ : \ g D(a) = 0 = gH(a) \ for \ all \ a \in \mathcal{A} \right\}.$$
The above discussion shows that $S(G) \subseteq \mathfrak{J}$. Our next task is to show that $\mathfrak{J}$ is an ideal of $\mathcal{B}$. Clearly, $\mathfrak{J}$ is a left ideal of $\mathcal{B}$. Let $b$ and $g$ be arbitrary elements of $\mathcal{B}$ and $\mathfrak{J}$, respectively. Since $G$ is surjective, there exists an element $a_1 \in \mathcal{A}$ such that $G(a_1) = b$. Using (\ref{2.8}) and (\ref{2.9}), we have the following:
$$0 = g D(a_1 a) = g H(a_1) F(a) + g G(a_1) D(a) = g b D(a)$$ and further $$0 = g D(a_1 a) = g D(a_1 )F(a) + g g(a_1 ) H(a) = gb H(a)$$ for all $a \in \mathcal{A}$. This implies that $\mathfrak{J}$ is an ideal of $\mathcal{B}$. Since $\mathcal{B}$ is a simple algebra, either $\mathfrak{J} = \{0\}$, which implies that $G$ is continuous, or $\mathfrak{J} = \mathcal{B}$ and in this case, we deduce that $D = H = 0$. From what has been said, we conclude that $F$ and $G$ are continuous mappings or $D = H = 0$.
\end{proof}

By similar arguments to those in the proof of the above theorem, we get the following result whose proof is left to the interested reader.

\begin{theorem} \label{12} 
Suppose that $\mathcal{A}$ is a Banach algebra, $\mathcal{B}$ is a simple Banach algebra and $D:\mathcal{A} \rightarrow \mathcal{B}$ is a continuous $(D, F, F, H)$-derivation such that $D(\mathcal{A}) \subseteq H(\mathcal{A})$ or $H(\mathcal{A}) \subseteq D(\mathcal{A})$ and $F$ is surjective. Then $F$ is continuous or $H$ and $D$ are identically zero.
\end{theorem}

\begin{corollary}\label{13} 
Let $\mathcal{A}$ be a Banach algebra and $\mathcal{B}$ be a simple Banach algebra.\\
(i) Suppose that $D:\mathcal{A} \rightarrow \mathcal{B}$ is both a continuous $(D, F, G, H)$-derivation and a continuous $(H, F, G, D)$-derivation such that $H$ is continuous and $F, G$ are surjective. If $F$ or $G$ is a discontinuous mapping, then $D$ and $H$ are identically zero.\\
(ii) Suppose that $D:\mathcal{A} \rightarrow \mathcal{B}$ is a continuous $(D, F, F, H)$-derivation such that $D(\mathcal{A}) \subseteq H(\mathcal{A})$ or $H(\mathcal{A}) \subseteq D(\mathcal{A})$. If $F:\mathcal{A} \rightarrow \mathcal{B}$ is a discontinuous, surjective, linear mapping, then $D$ and $H$ are identically zero. \end{corollary}

In the following, we present some conditions under which the linear mappings $F$ and $G$ associated with a $(D, F, G,H)$-derivation $D$ are
automatically continuous.

\begin{theorem} \label{14} Suppose that $\mathcal{A}$ is a Banach algebra, $\mathcal{B}$ is a semiprime Banach algebra and $D:\mathcal{A} \rightarrow \mathcal{B}$ is a $(D, F, G, H)$-derivation such that $D$ is surjective and $H$ is continuous (or $D$ is continuous and $H$ is surjective). Then $F$ and $G$ are continuous.
\end{theorem}

\begin{proof} Suppose that $D$ is surjective and $H$ is continuous. As in (2.6) we have $D(\mathcal{A})c = \{0\}$ for all $c \in S(F)$. Since $D$ is surjective and $\mathcal{B}$ is a semiprime algebra, $c = 0$. This means that $F$ is continuous. The continuity of $G$ is obtained via similar arguments to those employed in the proof of Theorem \ref{+++}. If we assume that $D$ is continuous and $H$ is surjective, then our claim is easily obtained by (2.7) and (2.9).
 \end{proof}

%\begin{theorem} \label{1} Let $\mathcal{A}$ be a unital normed algebra, let $\mathcal{B}$ be a normed algebra, let $\mathcal{M}$ be a Banach $\mathcal{B}$-module and let $D:\mathcal{A} \rightarrow \mathcal{M}$ be an $(E, F, G, H)$-derivation. Then each of the following conditions implies the continuity of $D$.\\
%(i) $F$ and $H$ are continuous;\\
%(ii) $G$ and $E$ are continuous.
%\end{theorem}

%\begin{proof} (i) For every $a \in \mathcal{A}$, we have the following expressions:
%\begin{align*}
%\|D(a)\| = \|E(\textbf{1}) F(a) + G(\textbf{1}) H(a)\| \leq \|E(\textbf{1})\| \|F\|\|a\| + \|G(\textbf{1})\| \|H\|\|a\|,
%\end{align*}
%and so
%\begin{align*}
%\|D(a)\| \leq \big(\|E(\textbf{1})\| \|F\| + \|G(\textbf{1})\|\|H\|\big)\|a\|.
%\end{align*}
%It means that $D$ is continuous.\\

%(ii) Using a similar argument, one can get the result.
%\end{proof}

Let $\mathcal{A}$ be an algebra and  $\mathcal{A}_1$ denote the set of all pairs $(a, \alpha)$, $a \in \mathcal{A}$, $\alpha \in \mathbb{C}$, that is, $\mathcal{A}_1 = \mathcal{A} \bigoplus  \mathbb{C}$. Then $\mathcal{A}_1$ becomes an algebra if the linear space operations and multiplication are defined by
$(a, \alpha) + (b, \beta) = (a + b, \alpha + \beta)$, $\mu(a, \alpha) = (\mu a, \mu \alpha)$
and $(a, \alpha)(b, \beta) = (ab + \alpha b + \beta a, \alpha \beta)$
for $a, b \in \mathcal{A}$ and $\alpha, \beta \in \mathbb{C}$. A simple calculation shows that the element
$\textbf{1} = (0, 1) \in \mathcal{A}_1$ is the identity element of $\mathcal{A}_1$. %Moreover, the mapping $a \mapsto (a, 0)$ is an algebra isomorphism of $\mathcal{A}$ onto an ideal of codimension one in $\mathcal{A}^{\sharp}$. Obviously, $\mathcal{A}^{\sharp}$ is commutative if and only if $\mathcal{A}$ is commutative.

\begin{theorem}
 Let $\mathcal{A}$ and $\mathcal{B}$ be two algebras and  $\delta:\mathcal{A} \rightarrow \mathcal{B}$ be a linear mapping. Let $D: \mathcal{A}_1 \rightarrow \mathcal{B}_1$ be defined as $D(a, \alpha) = (\delta(a), \alpha)$ and  $H: \mathcal{A}_1 \rightarrow \mathcal{B}_1$ be a linear mapping such that $H(\textbf{1}) = \textbf{1}$. Then $D$ is a $(D, F, G, H)$-derivation for some $F, G$ if and only if $\delta$ is a homomorphism. Furthermore, if $D$ is a $(D, F, F, H)$-derivation, then $D = H$.
\end{theorem}

\begin{proof} 
Let $\delta:\mathcal{A} \rightarrow \mathcal{B}$ be a linear mapping and let $D: \mathcal{A}_1 \rightarrow \mathcal{B}_1$ defined by $D(a, \alpha) = (\delta(a), \alpha)$ be a $(D, F, G, H)$-derivation for some $F, G$. We know that $\textbf{1} = (0, 1)$ is the identity element of $\mathcal{A}_1$ and clearly, $D(\textbf{1}) = \textbf{1} = H(\textbf{1})$. So
\begin{align}\label{2.10}
\textbf{1} = D(\textbf{1}) = D(\textbf{1}) F(\textbf{1}) + G(\textbf{1}) H(\textbf{1}) = F(\textbf{1}) + G(\textbf{1}).
\end{align}
Denoting every element $(a, \alpha) \in \mathcal{A}_1$ by $a_{\alpha}$, we have the following:
\begin{align*}
D(a_{\alpha}) & = D(a_{\alpha})F(\textbf{1}) + G(a_{\alpha})H(\textbf{1}) \\ & = D(a_{\alpha})F(\textbf{1}) + G(a_{\alpha}) \\ & = D(a_{\alpha})(\textbf{1} - G(\textbf{1})) + G(a_{\alpha}) \\ & = D(a_{\alpha}) - D(a_{\alpha})G(\textbf{1}) + G(a_{\alpha}),
\end{align*}
which means that
\begin{align}\label{2.11}
G(a_{\alpha}) = D(a_{\alpha})G(\textbf{1}), \ \ \ \ \ \ \ \ \ \ \ \ \ \ \ (a_{\alpha} \in \mathcal{A}_1).
\end{align}
Similarly, we can deduce that
\begin{align}\label{2.12}
F(a_{\alpha}) = D(a_{\alpha}) - G(\textbf{1})H(a_{\alpha}), \ \ \ \ \ \ \ (a_{\alpha} \in \mathcal{A}_1).
\end{align}
Using (\ref{2.11}) and (\ref{2.12}), we have the following:
\begin{align*}
D(a_{\alpha}b_{\beta}) & = D(a_{\alpha})F(b_{\beta}) + G(a_{\alpha})H(b_{\beta}) \\ & = D(a_{\alpha})\big( D(b_{\beta}) - G(\textbf{1})H(b_{\beta})\big)+ D(a_{\alpha})G(\textbf{1})H(b_{\beta}) \\ & = D(a_{\alpha})D(b_{\beta}),
\end{align*}
which means that $D$ is a homomorphism on $\mathcal{A}_1$ and  so is $\delta$ on $\mathcal{A}$. Conversely, suppose that $\delta$ is a homomorphism on $\mathcal{A}$. A straightforward calculation shows that $D$ is also a homomorphism on $\mathcal{A}_1$. Thus we have
\begin{align*}
D(a_{\alpha}b_{\beta}) & = D(a_{\alpha})D(b_{\beta}) \\ & = D(a_\alpha)\frac{D(b_\beta)}{2} + \frac{D(a_\alpha)}{2}D(b_\beta).
\end{align*}
Considering $F = G = \frac{D}{2}$, we see that $D$ is a $(D, F, F, D)$-derivation on $\mathcal{A}_1$. Now, suppose $D$ is a $(D, F, F, H)$-derivation. We get from (\ref{2.10}) that $F(\textbf{1}) = \frac{\textbf{1}}{2}$ and it follows from (\ref{2.11}) that
\begin{align}\label{2.13}
F(a_{\alpha}) = \frac{D(a_{\alpha})}{2}, \ \ \ \ \ \ \ \ \ \ \ \ \ \ \ \ \ \ \ \ (a_{\alpha} \in \mathcal{A}_1),
\end{align}
and it also follows from (\ref{2.12}) and (\ref{2.13}) that
\begin{align*}
\frac{D(a_{\alpha})}{2} = D(a_{\alpha}) - \frac{H(a_{\alpha})}{2}, \ \ \ \ \ \ \ \ \ \ \ \ \ \ \ (a_{\alpha} \in \mathcal{A}_1).
\end{align*}
The previous equality implies that $D = H$, as required.
\end{proof}

\begin{theorem}
 Let $\mathcal{A}$ and $\mathcal{B}$ be two topological algebras and  $D: \mathcal{A} \rightarrow \mathcal{B}$ be a $(D, F, G, H)$-derivation such that $D(\textbf{1}) = H(\textbf{1}) = \textbf{1}$. Suppose that every homomorphism from $\mathcal{A}$ into $\mathcal{B}$ is continuous. Then $D$ and $G$ are continuous mappings and $G$ is a right $D$-centralizer.
\end{theorem}

\begin{proof} 
Using the proof of the previous theorem, we get
\begin{align}\label{2.14}
G(a) = D(a)G(\textbf{1}), \ \ \ \ \ \ \ \ \ \ \ \ \ a \in \mathcal{A},
\end{align} and also $D$ is a homomorphism. Since we have all the conditions under which every homomorphism from $\mathcal{A}$ into $\mathcal{B}$ is continuous, $D$ is continuous and it follows from (\ref{2.14}) that so is $G$. Moreover, we have
\begin{align*}
G(ab) = D(ab)G(\textbf{1}) = D(a)D(b)G(\textbf{1}) = D(a)G(b), \ \ \ \ \ \ \ a, b \in \mathcal{A},
\end{align*}
which means that $G$ is a right $D$-centralizer.
\end{proof}

%In the following theorem, we present a characterization of strongly generalized derivations of order one.

\begin{theorem}\label{3} 
Suppose that $\mathcal{A}$ and $\mathcal{B}$ are two unital, normed algebras such that $\mathcal{B}$ is a domain and assume that $D : \mathcal{A } \rightarrow \mathcal{B}$ is a $(D, F, G, H)$-derivation such that $D(\textbf{1}) = H(\textbf{1}) \neq 0$. If there exists a sequence $\{a_n\} \subseteq \mathcal{A}$ such that the sequences $\{D(a_n)\}, \{F(a_n)\}, \{G(a_n)\}$ and $\{H(a_n)\}$ are convergent to a nonzero element $a_0$. Then $G = D = \frac{F + H}{2}$. In particular, if $D = H$, then $D = F = G$.
\end{theorem}

\begin{proof} 
We know that $D(ab) = D(a) F(b) + G(a)H(b)$ for all $a, b \in \mathcal{A}$. So we have
\begin{align*}
a_0 &= lim_{n \rightarrow \infty}D(a_n) \\ &= lim_{n \rightarrow \infty}(D(a_n)F(\textbf{1}) + G(a_n)H(\textbf{1})) \\ & = a_0(F(\textbf{1}) + H(\textbf{1})) \\ & = a_0(F(\textbf{1}) + D(\textbf{1})).
\end{align*}
Since $\mathcal{B}$ is a domain and $a_0 \neq 0$, we see that $F(\textbf{1}) + D(\textbf{1}) = \textbf{1}$. We also have
\begin{align*}
D(\textbf{1}) &= D(\textbf{1})F(\textbf{1}) + G(\textbf{1})H(\textbf{1})\\ & =  D(\textbf{1})F(\textbf{1}) + G(\textbf{1})D(\textbf{1}) \\ & = D(\textbf{1})(\textbf{1} - D(\textbf{1})) + G(\textbf{1})D(\textbf{1}) \\ & = D(\textbf{1}) - D(\textbf{1})^2 + G(\textbf{1})D(\textbf{1})
\end{align*}
and consequently, $(G(\textbf{1}) - D(\textbf{1}))D(\textbf{1}) = 0$. Since $\mathcal{B}$ is a domain and $D(\textbf{1}) \neq 0$, we get that $G(\textbf{1}) = D(\textbf{1})$. Let $a$ be an arbitrary element of $\mathcal{A}$. We have
\begin{align*}
D(a) &= D(a)F(\textbf{1}) + G(a)H(\textbf{1}) \\ & = D(a)(\textbf{1} - D(\textbf{1})) + G(a)D(\textbf{1}) \\ & = D(a) - D(a)D(\textbf{1}) + G(a)D(\textbf{1}).
\end{align*}
Since $\mathcal{B}$ is a domain and $D(\textbf{1}) \neq 0$, we deduce that $G = D$. Using this fact, we see that
\begin{align*}
a_0 &= lim_{n \rightarrow \infty}D(a_n) \\ &= lim_{n \rightarrow \infty}(D(\textbf{1})F(a_n) + G(\textbf{1})H(a_n)) \\ & = 2D(\textbf{1})a_0 .
\end{align*}Since $\mathcal{B}$ is a domain, $D(\textbf{1}) = \frac{\textbf{1}}{2} = G(\textbf{1})$. Therefore, we have
\begin{align*}
D(a) &= D(\textbf{1})F(a) + G(\textbf{1})H(a) \\ & = \frac{F(a)}{2} + \frac{H(a)}{2}
\end{align*}
for all $a \in \mathcal{A}$. This implies that $D = \frac{F + H}{2}$. It is clear that if $D = H$, then we get that $G = D = F$, as desired.
\end{proof}

An immediate corollary reads as follows:

\begin{corollary}\label{3} Suppose that $\mathcal{A}$ and $\mathcal{B}$ are two unital, normed algebras such that $\mathcal{B}$ is a domain and assume that $D : \mathcal{A } \rightarrow \mathcal{B}$ is a $(D, F, G, H)$-derivation such that $D(\textbf{1}) = H(\textbf{1}) \neq 0$. If there exists an element $a_0 \in \mathcal{A}$ such that $D(a_0) = F(a_0) = G(a_0) = H(a_0) \neq 0$, then $G = D = \frac{F + H}{2}$. In particular, if $D = H$, then $D = F = G$.
\end{corollary}

For any $a \in \mathcal{A}$, we define the linear mappings $L_a, R_a : \mathcal{A} \rightarrow \mathcal{A}$ by $L_a(b) = ab$ and $R_a(b) = ba$ for all $b \in \mathcal{A}$. A straightforward verification shows that $R_b L_a = L_aR_b$, $\lambda L_a = L_{\lambda a}$ and $\lambda R_a = R_{\lambda a}$ for all $a, b \in \mathcal{A}$ and $\lambda \in \mathbb{C}$. Clearly, $L_a = R_a$ if and only if $a \in Z(\mathcal{A})$.

In the following theorem, we provide a characterization of the generalized $(\sigma, \tau)$-derivations using some functional equations. It is clear that a generalized $(\sigma, \tau)$-derivation can be considered as a strongly generalized derivation of order one.

\begin{theorem} 
\label{13} Let $\mathcal{A}$ and $\mathcal{B}$ be two unital rings and  $\sigma, \tau:\mathcal{A} \rightarrow \mathcal{B}$ be two homomorphisms such that $\sigma(\textbf{1}) = \tau(\textbf{1}) = \textbf{1}$. Let $f, g, h: \mathcal{A} \rightarrow \mathcal{M}$ be additive mappings satisfying $$f(ab) = g(a)\sigma(b) + \tau(a)h(b) = h(a)\tau(b) + \sigma(a) g(b)$$
for all $a, b \in \mathcal{A}$. Then $f$, $g$ and $h$ are generalized $(\sigma, \tau)$-derivations on $\mathcal{A}$. In particular, if $\sigma = \tau$, then all the mappings $f$, $g$ and $h$ are generalized $\sigma$-derivations associated with a $\sigma$-derivation.
\end{theorem}

\begin{proof} 
Putting $a = \textbf{1}$ and $b = \textbf{1}$, we have the following:
\begin{align*}
&f(a) = g(a) + \tau(a)h(\textbf{1}), \\ & f(b) = h(\textbf{1})\tau(b) + g(b),
\end{align*}
for all $a, b \in \mathcal{A}$. It means that
\begin{align*}
\tau(a)h(\textbf{1}) = h(\textbf{1})\tau(a), \ \ \ \ \ \ \ \ (a \in \mathcal{A}).
\end{align*}
Similarly, we obtain that
\begin{align}\label{2.16}
g(\textbf{1}) \sigma(a) = \sigma(a) g(\textbf{1}), \ \ \ \ \ \ \ \ (a \in \mathcal{A}).
\end{align}
Since $f(a) = g(a) + \tau(a)h(\textbf{1}) = h(a) + \sigma(a) g(\textbf{1})$ for all $a \in \mathcal{A}$, we get that
\begin{align}\label{2.17}
h(a) = g(a) + \tau(a)h(\textbf{1}) - \sigma(a) g(\textbf{1}), \ \ \ \ \ \ \ \ (a \in \mathcal{A}),
\end{align}
and further, we have $$g(a)\sigma(b) + \tau(a)h(b) = f(ab) = g(ab) + \tau(ab)h(\textbf{1}).$$ Hence, for all $a, b \in \mathcal{A}$, we have the following:
\begin{align*}
g(ab) & = g(a)\sigma(b) + \tau(a)h(b) - \tau(a)\tau(b)h(\textbf{1}) \\ & = g(a)\sigma(b) + \tau(a)(g(b) + \tau(b)h(\textbf{1}) - \sigma(b) g(\textbf{1})) - \tau(a)\tau(b)h(\textbf{1}) \ \ \ \ (see \ (\ref{2.17})) \\ & = g(a)\sigma(b) + \tau(a)(g(b) - \sigma(b) g(\textbf{1})),
\end{align*}
which means that
\begin{align}\label{2.18}
g(ab) = g(a)\sigma(b) + \tau(a)(g(b) - \sigma(b) g(\textbf{1})).
\end{align}
Define a mapping $D:\mathcal{A} \rightarrow \mathcal{M}$ by $D(a) = g(a) - \sigma(a) g(\textbf{1})$. We have the following:
\begin{align*}
D(ab) & = g(ab) - \sigma(ab) g(\textbf{1}) \\ & = g(a)\sigma(b) + \tau(a)g(b) - \tau(a) \sigma(b) g(\textbf{1}) - \sigma(a)\sigma(b) g(\textbf{1}) \\ & = (g(a) - \sigma(a) g(\textbf{1}))\sigma(b) + \tau(a)(g(b) - \sigma(b) g(\textbf{1})) \\ & = D(a)\sigma(b) + \tau(a) D(b),
\end{align*}
which means that $D$ is a $(\sigma, \tau)$-derivation. This fact with (\ref{2.18}) implies that $g$ is a generalized $(\sigma, \tau)$-derivation associated with the $(\sigma, \tau)$-derivation $D$. Indeed, we have $$ g(ab) = g(a)\sigma(b) + \tau(a)D(b).$$ Now, we define a mapping $\delta:\mathcal{A} \rightarrow \mathcal{M}$ by $\delta(a) = g(a) - \tau(a)g(\textbf{1}).$ Using (\ref{2.16}) and (\ref{2.18}), we have
\begin{align*}
g(ab) & = g(a)\sigma(b) + \tau(a)g(b) - \tau(a)\sigma(b) g(\textbf{1}) \\ & = (g(a) - \tau(a)g(\textbf{1}))\sigma(b) + \tau(a)g(b) \\ & = \delta(a)\sigma(b) + \tau(a)g(b).
\end{align*}
A straightforward verification shows that $$\delta(ab) = \delta(a)\sigma(b) + \tau(a) \delta(b)$$ for all $a, b \in \mathcal{A}$. We know that $g(ab) = g(a)\sigma(b) + \tau(a)(h(b) - \tau(b)h(\textbf{1}))$ and defining a mapping $d:\mathcal{A} \rightarrow \mathcal{M}$ by $d(a) = h(a) - \tau(a) h(\textbf{1})$, we have $g(ab) = g(a)\sigma(b) + \tau(a)d(b)$  for all $a, b \in \mathcal{A}$. On the other hand, we know that $g$ is a generalized $(\sigma, \tau)$-derivation associated with the $(\sigma, \tau)$-derivation $D$, i.e., $g(ab) = g(a)\sigma(b) + \tau(a)D(b)$ for all $a, b \in \mathcal{A}$. So we infer that $d = D$ and this means that $d$ is also a $(\sigma, \tau)$-derivation. Hence $h(a) = d(a) + \tau(a)h(\textbf{1})$ is a generalized $(\sigma, \tau)$-derivation associated with the $(\sigma, \tau)$-derivation $d = D$ and we see that $$h(ab) = D(a)\sigma(b) + \tau(a)h(b),$$ for all $a, b \in \mathcal{A}$. Furthermore, we know that $f(ab) = h(ab) + \sigma(ab) g(\textbf{1})$ for all $a, b \in \mathcal{A}$. Thus we have $$h(a)\tau(b) + \sigma(a) g(b) = f(ab) = h(ab) + \sigma(ab) g(\textbf{1})$$ for all $a, b \in \mathcal{A}$. So we can write that
\begin{align*}
h(ab) & = h(a)\tau(b) + \sigma(a)g(b) - \sigma(ab) g(\textbf{1}) \\ & = h(a)\tau(b) + \sigma(a)(g(b) - \sigma(b) g(\textbf{1})) \\ & = h(a)\tau(b) + \sigma(a)D(b).
\end{align*}
We know that the linear mapping $\delta = g - R_{g(\textbf{1})}\tau$ is a $(\sigma, \tau)$-derivation and further, $f = g + R_{h(\textbf{1})}\tau$. Using these two previous equations, we can thus deduce that
\begin{align}\label{2.19}
f = \delta + R_{g(\textbf{1})}\tau + R_{h(\textbf{1})}\tau = \delta + R_{f(\textbf{1})}\tau,
\end{align}
which means that $f$ is a generalized $(\sigma, \tau)$-derivation associated with the $(\sigma, \tau)$-derivation $\delta$. Moreover, since $g = D + R_{g(\textbf{1})}\sigma$, we see that
\begin{align}\label{2.20}
f = g + R_{h(\textbf{1})}\tau = D + R_{g(\textbf{1})}\sigma + R_{h(\textbf{1})}\tau.
\end{align}
Now, suppose that $\sigma = \tau$. Then we have $\delta = g - R_{g(\textbf{1})}\tau = g - R_{g(\textbf{1})}\sigma = D$. It follows from (\ref{2.19}) or (\ref{2.20}) that
\begin{align*}
f = D + R_{f(\textbf{1})}\sigma.
\end{align*}
So we have
\begin{align*}
f(ab) & = D(ab) + R_{f(\textbf{1})}\sigma(ab) \\ & = D(a) \sigma(b) + \sigma(a)D(b) + f(\textbf{1})\sigma(a)\sigma(b) \\ & = (D(a) + f(\textbf{1})\sigma(a))\sigma(b) + \sigma(a)D(b) \\ & = f(a)\sigma(b) + \sigma(a)D(b)
\end{align*}
for all $a, b \in \mathcal{A}$. Also, note that $f(\textbf{1}) \sigma(a) = \sigma(a)f(\textbf{1})$ for all $a \in \mathcal{A}$, since $\sigma = \tau$. One can easily deduce that
\begin{align*}
f(ab) & = D(a)\sigma(b) + \sigma(a)f(b)
\end{align*}
for all $a, b \in \mathcal{A}$. Therefore, all the mappings $f$, $g$ and $h$ are generalized $\sigma$-derivations associated with the $\sigma$-derivation $D$.
\end{proof}

\section{Conclusion}

Let $\mathcal{A}$ and $\mathcal{B}$ be two unital topological algebras and  $f$, $g$ and $h$ be mappings which satisfy the equations of Theorem \ref{13} and $\sigma (\textbf{1}) = \textbf{1}= \tau(\textbf{1})$. If we have all the conditions under which every $(\sigma, \tau)$-derivation is continuous, then the mappings $f$, $g$ and $h$ are continuous under the same conditions.

%%%%%%%%%%%%%%%%%%%%%%%%%%%%%%%%%%%%%%%%%%%%%%
\medskip

\section*{Declarations}

\medskip

\noindent \textbf{Availablity of data and materials}\newline
\noindent Not applicable.

\medskip

\noindent \textbf{Human and animal rights}\newline
\noindent We would like to mention that this article does not contain any studies
with animals and does not involve any studies over human being.

\medskip

\noindent \textbf{Conflict of interest}\newline
\noindent The authors declare that they have no competing interests.

\medskip

\noindent \textbf{Fundings} \newline
\noindent The authors declare that there is no funding available for this paper.

\medskip

\noindent \textbf{Authors' contributions}\newline
\noindent The authors equally conceived of the study, participated in its
design and coordination, drafted the manuscript, participated in the
sequence alignment, and read and approved the final manuscript. 

\medskip

{\footnotesize


\begin{thebibliography}{99}


\bibitem{B1} W. G. Bade,  P. C. Curtis, {\it Homomorphisms of commutative Banach algebras}, Amer.
J. Math., {\bf 82} (1960), 589--608.

\bibitem{bo} N. H. Bingham, A. J. Ostaszewski,  {\it Additivity, subadditivity and linearity: automatic continuity and quantifier weakening}, Indag.  Math. (N.S.), {\bf 29} (2018), no. 2, 687--713.

\bibitem{B2} O. Bratteli,  D. R. Robinson, {\it Operator Algebras and Quantum Statistical Mechanics, Vol. I}, Springer Verlag, New York, 1987.

\bibitem{B3} O. Bratteli,  D. R. Robinson, {\it Operator Algebras and Quantum Statistical Mechanics, Vol. II}, Springer Verlag, New York, 1997.



\bibitem{D}  H. G. Dales, {\it Banach Algebras and Automatic Continuity}, Math. Soc. Monographs, New Series, 24, Oxford University Press, Oxford, 2000.

\bibitem{D1} H. G. Dales, P. Aiena, J. Eschmeier, K. Laursen,  G. A. Willis. {\it Introduction to Banach Algebras, Operators and Harmonic Analysis}, Cambridge University Press, Cambridge, 2003.


\bibitem{E} O. Elchinger, K. Lundengård, A. Makhlouf,  S. Silvestrov, {\it Brackets with $(\tau,\sigma)$-derivations and $(p,q)$-deformations of Witt and Virasoro algebras}, Forum Math.,  \textbf{28} (2016), no. 4,  657--673.


\bibitem{EP18} A. B. A. Essaleh,  A. M. Peralta, {\it Linear maps on $C^{\ast}$-algebras which are derivations or
triple derivations at a point}, Linear Algebra Appl.,  \textbf{538} (2018), 1--21.

\bibitem{fh} J. J. Font, S. Hern\'{a}ndez,  {\it Automatic continuity and representation of certain linear isomorphisms between group algebras}, Indag.  Math. (N.S.), {\bf 6} (1995), no. 4, 397--409.


\bibitem{GP13} J. J. Garc\'{e}s,   A. M. Peralta, {\it Generalised triple homomorphisms and derivations}, Canad. J. Math.,  \textbf{65} (2013), no. 4,  783--807.

\bibitem{GH18} M. Gholampour, S. Hejazian, {\it Continuity of generalized derivations on $JB^{\ast}$-algebras},
Quaest. Math.,  \textbf{41} (2018), 227-–238.

\bibitem{H5}J. T. Hartwig, D. Larsson, S. D. Silvestrov, {\it Deformations of Lie algebras using $\sigma$-derivations}, J. Algebra,  \textbf{295} (2006), 314--361.

\bibitem{He} S. Hejazian, A. R. Janfada, M. Mirzavaziri, {\it Achivement of continuity of $(\phi,\psi)$-derivations},  Bull. Belg. Math. Soc. Simon Stevin,  \textbf{14} (2007), 641--652.

\bibitem{Hel}  M. Heller, T. Miller, L. Pysiak,  W. Sasin, {\it Generalized derivations and general relativity}, Canad. J. Phys., \textbf{91} (2013), no. 10, 757--763.

\bibitem{H4} A. Hosseini, {\it Automatic continuity of $(\delta, \varepsilon$)-double derivations on $C^{\ast}$-algebras}, U.P.B. Sci. Bull., Ser. A., {\bf 79} (2017), no. 3, 67--72.

\bibitem{H3} A. Hosseini, {\it A new proof of Singer-Wermer theorem with some results on $\{g, h\}$-derivations}, Int. J. Nonlinear Anal. Appl., {\bf 11} (2020), no. 1, 453--471.

\bibitem{H4} A. Hosseini, {\it A note on automatic continuity of $(\psi, \phi)$-derivations}, Rend. Circ. Mat. Palermo (2), {\bf 72} (2023), no. 1, 71--79.

\bibitem{Hu} C. Hou, Q. Ming, {\it Continuity of $(\alpha, \beta)$-derivations of operator algebras},  J.
Korean Math. Soc.,  \textbf{48} (2011), 823--835.

\bibitem{J} B. E. Johnson, A. M. Sinclair, {\it Continuity of derivations and a problem of
Kaplansky},  Amer. J. Math., \textbf{90} (1968), 1067--1073.

\bibitem{J87} B. E. Johnson, {\it Continuity of generalized homomorphisms}, Bull. London Math. Soc., \textbf{19} (1987), 67--71.

\bibitem{K} I. Kaplansky, {\it Functional Analysis, Some Aspects of Analysis and Probability}, Surveys in Applied Mathematics. Vol. 4, London, 1958. MR0101475 (21:286)

\bibitem{Ke} E. Kreyszig, {\it Introductory Functional Analysis with Applications}, John  Wiley and Sons, New York, 1979.



\bibitem{mir} M. Mirzavaziri,  E. Omidvar Tehrani, {\it $\delta$-Double derivations on $C^{\ast}$-algebras}, Bull. Iran. Math. Soc., \textbf{35}  (2009), no. 1, 147--154.

\bibitem{P} A. M. Peralta,  B. Russo, {\it Automatic continuity of derivations on $C^{*}$-algebras and $JB^{*}$-triples}, J. Algebra, \textbf{399} (2014), 960--977.

\bibitem{R}  J. R. Ringrose, {\it Automatic continuity of derivations of operator algebras}, J. London Math.
Soc., \textbf{5} (1972), 432--438.

\bibitem{S}S. Sakai, {\it On a conjecture of Kaplansky}, T$\hat{o}$hoku Math. J., \textbf{12} (1960), 31--33.

\bibitem{V} A. R. Villena, {\it  Automatic continuity in associative and nonassociative context},
Irish Math. Soc. Bull.,  \textbf{46} (2001), 43--76.


\end{thebibliography}
\end{document}